\documentclass{agnt}

\usepackage[T1]{fontenc}
\usepackage[utf8]{inputenc}

\usepackage{amsfonts}
\usepackage{amsmath}
\usepackage{amssymb}
\usepackage{amsthm}
\usepackage{enumerate}
\usepackage[hang,flushmargin]{footmisc}

\usepackage{graphicx}
\graphicspath{{eps/}}
\usepackage{epstopdf}
\usepackage{epsfig}
 
\usepackage{bm} 

\newtheorem{theorem}{Theorem}[section]
\newtheorem{lemma}{Lemma}[section]
\newtheorem{corollary}{Corollary}[section]

\title{Complex Divisor Functions}
\subject{11B05}{11A25}

\author{Colin Defant}
\address{University of Florida \\ Department of Mathematics}
\email{cdefant@ufl.edu}
\begin{document}

\maketitle

\begin{abstract} For any complex number $c$, let $\sigma_c\colon\mathbb N\rightarrow\mathbb C$ denote the divisor function defined by $\sigma_c(n)=\displaystyle{\sum_{d\vert n}d^c}$ for all $n\in\mathbb N$, and define $R(c)=\{\sigma_c(n)\in\mathbb C\colon n\in\mathbb N\}$ to be the range of $\sigma_c$. We study the basic topological properties of the sets $R(c)$. In particular, we determine the complex numbers $c$ for which $R(c)$ is bounded and determine the isolated points of the sets $R(c)$. In the third section, we find those values of $c$ for which $R(c)$ is dense in $\mathbb C$. We also prove some results and pose several open problems about the closures of the sets $R(c)$ when these sets are bounded.  
\end{abstract} 

\section{Introduction} 
Throughout this article, we will let $\mathbb N$, $\mathbb N_0$, and $\mathbb P$ denote the set of positive integers, the set of nonnegative integers, and the set of prime numbers, respectively. The lowercase letter $p$ will always denote a prime number, and $\nu_p(n)$ will denote the exponent of $p$ in the prime factorization of a positive integer $n$. The letter $\zeta$ will always denote the Riemann zeta function. Furthermore, for any nonzero complex number $z$, we let $\arg(z)$ denote the principal argument of $z$ with the convention that $-\pi<\arg(z)\leq\pi$. 
\par 
For any complex number $c$, the divisor function $\sigma_c\colon\mathbb{N}\rightarrow\mathbb C$ is the arithmetic function defined by $\sigma_c(n)=\displaystyle{\sum_{d\vert n}d^c}$ for all $n\in\mathbb N$. The function $\sigma_c$ is a multiplicative arithmetic function that satisfies $\sigma_c(p^\alpha)=1+p^c+\cdots+p^{\alpha c}$ for all primes $p$ and positive integers $\alpha$. Of course, if $p^c\neq 1$, then we may write $\sigma_c(p^{\alpha})=\displaystyle{\frac{p^{(\alpha+1)c}-1}{p^c-1}}$. Divisor functions are some of the most important functions in number theory; their appearances in various identities and applications are so numerous that we will not even attempt to list them. However, divisor functions other than $\sigma_1,\sigma_0$, and $\sigma_{-1}$ are rarely studied. Ramanujan did study the functions $\sigma_c$ for $c$ real, but his results were only published in the last two decades \cite{Andrews12,Ramanujan97}. Recently, the author \cite{Defant14} has studied the ranges of the functions $\sigma_c$ for real $c$ and has shown that there exists a constant $\eta\approx 1.8877909$ such that if $r\in(1,\infty)$, then the range of the function $\sigma_{-r}$ is dense in the interval $[1,\zeta(r))$ if and only if $r\leq\eta$. For any complex $c$, we will let $R(c)=\{\sigma_c(n)\colon n\in\mathbb{N}\}$ be the range of the function $\sigma_c$. In this article, we will study the basic topological properties of the sets $R(c)$ for various complex numbers $c$. More specifically, we will direct the bulk of our attention toward answering the following questions: 
\begin{enumerate} 
\item For which complex $c$ is $R(c)$ bounded? 
\item For which complex $c$ does $R(c)$ have isolated points?   
\item What can we tell about the closure $\overline{R(c)}$ of the set $R(c)$ for a given value of $c$? In particular, what are the values of $c$ for which $R(c)$ is dense in $\mathbb C$?  
\end{enumerate}    
We begin with a number of useful lemmas. Some of these lemmas not only aid in the proofs of later theorems, but also provide some basic yet interesting information that serves as a nice introduction to the sets $R(c)$. Henceforth, $c$ will denote a complex number with real part $a$ and imaginary part $b$. Recall that any complex number $z$ may be written as \[z=u+iv=re^{i\theta}=r(\cos\theta+i\sin\theta)\] so that $u=r\cos\theta$, $v=r\sin\theta$, $r=\sqrt{u^2+v^2}$, and $\tan\theta=v/u$. Therefore, for any positive real number $x$, we have \begin{equation}\label{EqEdit1}
\vert 1+x^c\vert^2=\vert 1+e^{c\log x}\vert^2=(1+x^a\cos(b\log x))^2+x^{2a}\sin^2(b\log x)
\end{equation} \[=1+2x^a\cos(b\log x)+x^{2a}\] and \begin{equation}\label{EqEdit2}
\tan\left(\arg(1+x^c)\right)=\frac{x^a\sin(b\log x)}{1+x^a\cos(b\log x)}.
\end{equation}
\begin{lemma} \label{Lem1.1}
For any $n\in\mathbb{N}$, $\sigma_{\bar c}(n)=\overline{\sigma_c(n)}$.  
\end{lemma} 
\begin{proof} 
We have \[\sigma_{\bar{c}}(n)=\sum_{d\vert n}d^{a-bi}=\sum_{d\vert n}d^ad^{-bi}=\sum_{d\vert n}\overline{d^a}\cdot\overline{d^{bi}}=\overline{\sum_{d\vert n}d^{a+bi}}=\overline{\sigma_c(n)}.\] 
\end{proof}
Lemma \ref{Lem1.1} tells us that $R(\overline c)$ is simply the reflection of the set $R(c)$ about the real axis. In many situations, this simple but useful lemma allows us to restrict our attention to complex numbers $c$ in the upper half plane and then use symmetry to deduce similar results for values of $c$ in the lower half-plane. 
\begin{lemma} \label{Lem1.2}
We have $0\in R(c)$ if and only if $a=0$ and $b=q\displaystyle{\frac{\pi}{\log p}}$ for some prime $p$ and some rational number $q$ that is not an even integer. 
\end{lemma} 
\begin{proof} 
First, suppose $a=0$ and $\displaystyle{b=q\frac{\pi}{\log p}}$, where $p$ is a prime and $q$ is a rational number that is not an even integer. As $q$ is not an even integer, $p^c=p^{bi}=e^{q\pi i}\neq 1$. We may write $q=\displaystyle{\frac{\ell}{m}}$ for some nonzero integers $\ell$ and $m$ with $m>0$. Then
\[\sigma_c(p^{2m-1})=\frac{p^{2mc}-1}{p^c-1}=\frac{p^{2mbi}-1}{p^c-1}=\frac{e^{2\ell\pi i}-1}{p^c-1}=0,\] so $0\in R(c)$. 
\par 
Conversely, suppose $0\in R(c)$. Then there exists some $n\in\mathbb N$ with $\sigma_c(n)=0$. Clearly $n>1$, so we may let $n=p_1^{\alpha_1}\cdots p_r^{\alpha_r}$ be the canonical prime factorization of $n$. Then $0=\sigma_c(n)=\sigma_c(p_1^{\alpha_1})\cdots\sigma_c(p_r^{\alpha_r})$, so $\sigma_c(p_i^{\alpha_i})=0$ for some $i\in\{1,\ldots,r\}$. Let $p=p_i$ and $\alpha=\alpha_i$. We know that $p^c\neq 1$ because, otherwise, we would have $\sigma_c(p^\alpha)=1+p^c+\cdots+p^{\alpha c}=\alpha+1\neq 0$. Therefore, $0=\sigma_c(p^{\alpha})=\displaystyle{\frac{p^{(\alpha+1)c}-1}{p^c-1}}$, so $p^{(\alpha+1)c}=1$. Now, $p^{(\alpha+1)c}=p^{(\alpha+1)a}e^{(\alpha+1)b(\log p) i}$ so we must have $p^{(\alpha+1)a}=1$ and $e^{(\alpha+1)b(\log p)i}=1$. Consequently, $a=0$ and $b=\displaystyle{\frac{2k\pi}{(\alpha+1)\log p}}$ for some integer $k$. Letting $q=\displaystyle{\frac{2k}{\alpha+1}}$, we see that $b$ has the desired form. Finally, $q$ is not an even integer because $e^{q\pi i}=p^c\neq 1$.  
\end{proof} 
\begin{lemma} \label{Lem1.3}
Suppose $a=0$ and $b\neq 0$. Let $\Psi(c)=\{\sigma_c(p)\colon p\in\mathbb P\}$, and let $C$ be the circle $\{1+z\in\mathbb C\colon\vert z\vert=1\}$. Then $\Psi(c)$ is a dense subset of $C$. 
\end{lemma} 
\begin{proof} 
By Lemma \ref{Lem1.1}, it suffices to prove our claim in the case $b>0$. Furthermore, because $\sigma_c(p)=1+p^c$ for all primes $p$, it suffices to show that the set $\Psi'(c)=\{p^c\colon p\in\mathbb P\}$ is a dence subset of the circle $C'=\{z\in\mathbb C\colon\vert z\vert=1\}$. We know that every point in $\Psi'(c)$ lies on the circle $C'$ because $\vert p^c\vert=p^a=1$ for all primes $p$. Now, choose some $z\in C'$ and some $\epsilon>0$. We may write $z=e^{i\varphi}$ for some $\varphi\in(-\pi,\pi]$. We wish to show that there exists a prime $p$ such that $\vert\arg(p^c)-\varphi+2t\pi\vert<\epsilon$ for some integer $t$. Equivalently, we need to show that there exists a prime $p$ and a positive integer $n$ such that $\varphi+2n\pi-\epsilon<b\log p<\varphi+2n\pi+\epsilon$ (this is because $\arg(p^c)=b\log p+2\pi m$ for some $m\in\mathbb Z$). Setting $\lambda=e^{(\varphi-\epsilon)/b}$, $\mu=e^{2\pi/b}$, and $\delta=e^{2\epsilon/b}$, we may rewrite these inequalities as $\lambda\mu^n<p<\lambda\mu^n\delta$. It follows from the well-known fact that $\displaystyle{\lim_{k\rightarrow\infty}\frac{p_{k+1}}{p_k}=1}$ that such a prime $p$ is guaranteed to exist for sufficiently large $n$ (here, we let $p_i$ denote the $i^{th}$ prime number). 
\end{proof} 
\begin{lemma} \label{Lem1.4}
If $a>0$, then $\vert\sigma_c(n)\vert\geq\displaystyle{\prod_{p<2^{1/a}}(p^a-1)}$ for all $n\in\mathbb N$.
\end{lemma} 
\begin{proof} 
Suppose $a>0$. For any prime $p$ and positive integer $\alpha$ we have \[\vert\sigma_c(p^\alpha)\vert=\left\lvert\frac{p^{(\alpha+1)c}-1}{p^c-1}\right\rvert\geq\frac{\vert p^{(\alpha+1)c}\vert-1}{\vert p^c\vert+1}=\frac{p^{(\alpha+1)a}-1}{p^a+1}\geq\frac{p^{2a}-1}{p^a+1}=p^a-1.\] Therefore, for any $n\in\mathbb N$, \[\vert\sigma_c(n)\vert=\left[\prod_{\substack{p\vert n \\ p<2^{1/a}}}\left\lvert\sigma_c\left(p^{\nu_p(n)}\right)\right\rvert\right]\left[\prod_{\substack{p\vert n \\ p\geq 2^{1/a}}}\left\lvert\sigma_c\left(p^{\nu_p(n)}\right)\right\rvert\right]\] \[\geq\left[\prod_{\substack{p\vert n \\ p<2^{1/a}}}(p^a-1)\right]\left[\prod_{\substack{p\vert n \\ p\geq 2^{1/a}}}(p^a-1)\right]\geq\prod_{\substack{p\vert n \\ p<2^{1/a}}}(p^a-1)\geq\prod_{p<2^{1/a}}(p^a-1).\]
\end{proof} 
In the third question that we posed above, we asked if we could find the values of $c$ for which $R(c)$ is dense in $\mathbb C$. Lemma \ref{Lem1.4} gives us an immediate partial answer to this question. If $a>0$, then $R(c)$ cannot be dense in $\mathbb C$ because there is a neighborhood of $0$ of radius $\displaystyle{\prod_{p<2^{1/a}}(p^a-1)}$ that contains no elements of $R(c)$. We will see in Theorem \ref{Thm2.2} that, in some sense, $R(c)$ is very far from being dense when $a>0$.  
\par 
The following lemma simply transforms an estimate due to Rosser and Shoenfeld into a slightly weaker inequality which is more easily applicable to our needs. 
\begin{lemma} \label{Lem1.5}
If $285\leq y<x$, then \[\prod_{p\in[y,x]}\left(1-\frac{1}{p}\right)<\frac{\log y}{\log x}+\frac{2}{\log^2y}.\] 
\end{lemma}
\begin{proof} 
Rosser and Shoenfeld's estimate \cite[Theorem 7]{Rosser62} states that if $x\geq 285$, then \[\frac{e^{-\gamma}}{\log x}\left(1-\frac{1}{2\log^2x}\right)<\prod_{p\leq x}\left(1-\frac{1}{p}\right)<\frac{e^{-\gamma}}{\log x}\left(1+\frac{1}{2\log^2x}\right),\] where $\gamma$ is the Euler-Mascheroni constant. Therefore, if $285\leq y<x$, then \[\prod_{p\in[y,x]}\left(1-\frac{1}{p}\right)<\left(\frac{e^{-\gamma}}{\log x}\left(1+\frac{1}{2\log^2x}\right)\right)\left(\frac{e^{-\gamma}}{\log y}\left(1-\frac{1}{2\log^2y}\right)\right)^{-1}\] \[=\frac{\log y}{\log x}\frac{1+1/(2\log^2x)}{1-1/(2\log^2y)}<\frac{\log y}{\log x}\frac{1+1/(2\log^2y)}{1-1/(2\log^2y)}<\frac{\log y}{\log x}\frac{\log^2y}{\log^2y-1}\] \[<\frac{\log y}{\log x}+\frac{2}{\log x\log y}<\frac{\log y}{\log x}+\frac{2}{\log^2y}.\] 
\end{proof} 
\begin{lemma} \label{Lem1.6}
Suppose $a\geq -1$ and $b>0$. Fix some $\beta\in\left(0,\displaystyle{\frac{\pi}{4}}\right)$. For each nonnegative integer $k$, let $\displaystyle{G_k=\left[e^{(2k\pi-\beta)/b},e^{(2k\pi+\beta)/b}\right]}$. Let $\displaystyle{\mathcal G=\bigcup_{k=0}^{\infty}G_k}$. If $x\in \mathcal G$, then $-\beta\leq \arg(x^c)\leq\beta$ and 
\[\vert 1+x^c\vert\geq\sqrt{1+2x^a\cos\beta+x^{2a}}.\] 
In addition, \[\sum_{p\in \mathcal{G}}\log\vert 1+p^c\vert=\infty.\] 
\end{lemma} 
\begin{proof} 
Suppose $x\in \mathcal{G}$. Then $x^c=x^ae^{ib\log x}$, so $\arg(x^c)=\arg(e^{ib\log x})$. By the definition of $\mathcal G$, $b\log x\in[2k\pi-\beta,2k\pi+\beta]$ for some nonnegative integer $k$. Therefore, $-\beta\leq\arg(e^{ib\log x})\leq\beta$. This implies that $\cos(b\log x)\geq\cos\beta$, so it follows from \eqref{EqEdit1} that \[\vert1+x^c\vert\geq\sqrt{1+2x^a\cos\beta+x^{2a}}.\]  
\par 
We now wish to prove that $\displaystyle{\sum_{p\in \mathcal G}\log\vert 1+p^c\vert=\infty}$. This sum makes sense (the order of the summands is immaterial) because all summands are positive by the preceding inequality. For sufficiently large $k$, we may use Lemma \ref{Lem1.5} to write \[\prod_{p\in G_k}\frac{p}{p-1}=\left(\prod_{p\in G_k}\left(1-\frac{1}{p}\right)\right)^{-1}>\left(\frac{2k\pi-\beta}{2k\pi+\beta}+\frac{2b^2}{(2k\pi-\beta)^2}\right)^{-1}\] \[=\left(1-\frac{2\beta}{2k\pi+\beta}+\frac{2b^2}{(2k\pi-\beta)^2}\right)^{-1}>\left(1-\frac{\beta}{2k\pi}+\frac{2b^2}{(2k\pi-\beta)^2}\right)^{-1}\] \[>\left(1-\frac{\beta}{7k}\right)^{-1}=1+\frac{\beta}{7k-\beta}.\] Also, if $p\geq 5$, then \[1+2p^a\cos\beta+p^{2a}>1+\frac{2\cos\beta}{p}\geq1+\frac{2\cos(\pi/4)}{p}>\frac{p}{p-1}.\]  Hence, for sufficiently large $k$, we have \[\sum_{p\in G_k}\log\vert 1+p^c\vert\geq\sum_{p\in G_k}\log\sqrt{1+2p^a\cos\beta+p^{2a}}>\sum_{p\in G_k}\log\sqrt{\frac{p}{p-1}}\] \[=\frac{1}{2}\log\prod_{p\in G_k}\frac{p}{p-1}>\frac{1}{2}\log\left(1+\frac{\beta}{7k-\beta}\right).\] The desired result then follows from the fact that \[\sum_{k=1}^\infty\frac{1}{2}\log\left(1+\frac{\beta}{7k-\beta}\right)=\infty.\] 
\end{proof}
\begin{lemma} \label{Lem1.7}
Suppose $-1\leq a<0$ and $b>0$. Fix some $\beta\in\left(0,\displaystyle{\frac{\pi}{4}}\right)$. For each nonnegative integer $k$, let $\displaystyle{H_k=\left[e^{((2k-1)\pi-\beta)/b},e^{((2k-1)\pi+\beta)/b}\right]}$. Let $\displaystyle{\mathcal H=\bigcup_{k=1}^{\infty}H_k}$. If $x\in \mathcal H$, then \[\vert 1+x^c\vert\leq\sqrt{1-2x^a\cos\beta+x^{2a}}\] and either $\arg(x^c)\leq-\pi+\beta$ or $\arg(x^c)\geq\pi-\beta$. 
In addition, \[\sum_{p\in \mathcal{H}}\log\vert 1+p^c\vert=-\infty.\] 
\end{lemma}
\begin{proof} 
The proof is quite similar to that of Lemma \ref{Lem1.6}. Suppose $x\in \mathcal{H}$. Then $x^c=x^ae^{ib\log x}$, so $\arg(x^c)=\arg(e^{ib\log x})$. By the definition of $\mathcal H$, \\ $b\log x\in[(2k-1)\pi-\beta,(2k-1)\pi+\beta]$ for some positive integer $k$. Therefore, either $\arg(x^c)\leq-\pi+\beta$ or $\arg(x^c)\geq\pi-\beta$. This implies that \\ $\cos(b\log x)\leq-\cos\beta$, so we see from \eqref{EqEdit1} that \[\vert1+x^c\vert\leq\sqrt{1-2x^a\cos\beta+x^{2a}}.\]  
\par 
We now show that \[\sum_{p\in \mathcal H}\log\vert 1+p^c\vert=-\infty.\] First, we need to check that the order of the summands in this summation does not matter. Because $a>0$ and $\beta\in\displaystyle{\left(0,\frac{\pi}{4}\right)}$, we have $p^a<1<2\cos\beta$ and, therefore, $-2p^a\cos\beta+p^{2a}<0$ for all $p\in\mathcal H$. Consequently, \[\log\vert 1+p^c\vert\leq\log\sqrt{1-2p^a\cos\beta+p^{2a}}<0\] for all $p\in\mathcal H$. In fact, if $p$ is sufficiently large, then \begin{equation}\label{EqEdit6} 1-2p^a\cos\beta+p^{2a}<1-p^a\leq 1-\frac{1}{p}.
\end{equation} As all summands are negative, their order does not matter. Using Lemma \ref{Lem1.5}, we see that if $k$ is sufficiently large, then \[\prod_{p\in H_k}\left(1-\frac{1}{p}\right)<\frac{(2k-1)\pi-\beta}{(2k-1)\pi+\beta}+\frac{2b^2}{((2k-1)\pi-\beta)^2}\] \[=1-\frac{2\beta}{(2k-1)\pi+\beta}+\frac{2b^2}{((2k-1)\pi-\beta)^2}<1-\frac{\beta}{4k}.\] Hence, for sufficiently large $k$, we may use \eqref{EqEdit6} to find that \[\sum_{p\in H_k}\log\vert 1+p^c\vert\leq\sum_{p\in H_k}\log\sqrt{1-2p^a\cos\beta+p^{2a}}<\sum_{p\in H_k}\log\sqrt{1-\frac{1}{p}}\] \[=\frac{1}{2}\log\prod_{p\in H_k}\left(1-\frac{1}{p}\right)<\frac{1}{2}\log\left(1-\frac{\beta}{4k}\right).\] We now obtain the desired result from the fact that \[\sum_{k=1}^\infty\frac{1}{2}\log\left(1-\frac{\beta}{4k}\right)=-\infty.\] 
\end{proof}   
\begin{lemma} \label{Lem1.8} 
Suppose $b>0$. For each nonnegative integer $k$, let \\ $J_k=\left[e^{\left(\left(2k+\frac{1}{4}\right)\pi\right)/b},e^{\left(\left(2k+\frac{3}{4}\right)\pi\right)/b}\right]$, and let $\mathcal J=\displaystyle{\bigcup_{k=0}^\infty J_k}$. If $x\in\mathcal J$, then \[\frac{\pi}{4}\leq\arg(x^c)\leq\frac{3\pi}{4}.\] Also, $\displaystyle{\sum_{p\in\mathcal J}\frac 1p=\infty}$.
\end{lemma} 
\begin{proof} 
If $x\in\mathcal J$, then $b\log x\in\left[\left(2k+\dfrac 14\right)\pi,\left(2k+\dfrac 34\right)\pi\right]$ for some nonnegative integer $k$, so $\dfrac{\pi}{4}\leq\arg(x^c)\leq\dfrac{3\pi}{4}$. For $286\leq u<v$, we may use Theorem 5 in \cite{Rosser62} to write \[\sum_{p\in[u,v]}\frac 1p\geq\sum_{p\leq v}\frac 1p-\sum_{p\leq u}\frac 1p\] \[>(\log\log v+B-1/(2\log^2v))-(\log\log u+B+1/(2\log^2u))\] \[=\log\left(\frac{\log v}{\log u}\right)-\frac{1}{2\log^2u}-\frac{1}{2\log^2v},\] where $B$ is a real constant. Therefore, for sufficiently large $k$, we have \[\sum_{p\in J_k}\frac 1p>\log\left(\frac{(2k+\frac 34)\pi/b}{(2k+\frac 14)\pi/b}\right)-\frac{b^2}{2\pi^2(2k+\frac 14)^2}-\frac{b^2}{2\pi^2(2k+\frac 34)^2}\] \[=\log\left(1+\frac{2}{8k+1}\right)-\left[\frac{b^2}{2\pi^2(2k+\frac 14)^2}+\frac{b^2}{2\pi^2(2k+\frac 34)^2}\right].\] Because \[\sum_{k=0}^\infty\left(\frac{b^2}{2\pi^2(2k+\frac 14)^2}+\frac{b^2}{2\pi^2(2k+\frac 34)^2}\right)\] converges and \[\sum_{k=0}^\infty\log\left(1+\frac{2}{8k+1}\right)=\infty,\] it follows that \[\sum_{p\in\mathcal J}\frac 1p=\sum_{k=0}^\infty\sum_{p\in J_k}\frac 1p=\infty.\] 
\end{proof}  
\begin{lemma}\label{Lem1.9} 
Let $\beta\in\displaystyle{\left(0,\frac{\pi}{4}\right)}$, and let $x>0$ be a real number such that $x^a<\sqrt2-1$. If $\arg(x^c)>\pi-\beta$ or $\arg(x^c)<-\pi+\beta$, then \[\vert\arg(1+x^c)\vert<\frac{\sin\beta}{2-\sqrt2}x^a.\]
\end{lemma} 
\begin{proof} 
Let $f(t)=\dfrac{x^a\sin t}{1-x^a\cos t}$, and let \[\varphi=\begin{cases} \pi-\arg(x^c), & \mbox{if } \arg(x^c)>\pi-\beta; \\ \pi+\arg(x^c), & \mbox{if } \arg(x^c)<-\pi+\beta. \end{cases}\] Note that $0<\varphi<\beta$, $\sin\varphi=\vert\sin(b\log x)\vert$, and $\cos\varphi=-\cos(b\log x)$. Also, for all $t\in\displaystyle{\left(0,\frac{\pi}{4}\right)}$, we have \[f'(t)=\frac{x^a(\cos t-x^a)}{(1-x^a\cos t)^2}>\frac{x^a(\cos(\pi/4)-x^a)}{(1-x^a\cos t)^2}>0\] because $x^a<\sqrt2-1<\cos(\pi/4)$. This implies that $f(\varphi)<f(\beta)$. Using \eqref{EqEdit2} along with the inequality $\vert\tan^{-1}\theta\vert\leq\vert\theta\vert$, which is valid for all real $\theta$, we find that \[\vert\arg(1+x^c)\vert=\left\lvert\tan^{-1}\left(\frac{x^a\sin(b\log x)}{1+x^a\cos(b\log x)}\right)\right\rvert\leq\left\lvert\frac{x^a\sin(b\log x)}{1+x^a\cos(b\log x)}\right\rvert\] \[=\frac{x^a\vert\sin(b\log x)\vert}{1+x^a\cos(b\log x)}=\frac{x^a\sin\varphi}{1-x^a\cos\varphi}=f(\varphi)<f(\beta)\] \[=\frac{x^a\sin \beta}{1-x^a\cos \beta}<\frac{x^a\sin \beta}{1-(\sqrt2-1)}=\frac{\sin \beta}{2-\sqrt2}x^a.\] 
\end{proof} 
\section{Boundedness and Isolated Points}
It turns out that questions $1$ and $2$ posed in the introduction are not too difficult to handle, so we will give complete answers to them in this section. 
\begin{theorem} \label{Thm2.1}
The set $R(c)$ is bounded if and only if $a<-1$.
\end{theorem} 
\begin{proof} 
If $a<-1$, then $R(c)$ is bounded because \[\vert\sigma_c(n)\vert=\left\lvert\sum_{d\vert n}d^c\right\rvert\leq\sum_{d\vert n}\vert d^c\vert=\sum_{d\vert n}d^a<\sum_{k=1}^\infty k^a=\zeta(-a)\] for all $n\in\mathbb{N}$. Now, suppose $a\geq -1$. By Lemma \ref{Lem1.1}, we see that it suffices to prove the result for $b\geq 0$. If $b=0$, then \[\lim_{x\rightarrow\infty}\sigma_c\left(\prod_{p\leq x}p\right)=\lim_{x\rightarrow\infty}\prod_{p\leq x}\sigma_c(p)=\lim_{x\rightarrow\infty}\prod_{p\leq x}\left(1+p^a\right)=\infty,\] so $R(c)$ cannot be bounded. If $b>0$, then the proof follows from Lemma \ref{Lem1.6} because \[\lim_{x\rightarrow\infty}\left\lvert\sigma_c\left(\prod_{\substack{p\in\mathcal G \\ p\leq x}}p\right)\right\rvert=\lim_{x\rightarrow\infty}\prod_{\substack{p\in\mathcal G \\ p\leq x}}\vert\sigma_c(p)\vert=\lim_{x\rightarrow\infty}\prod_{\substack{p\in\mathcal G \\ p\leq x}}\vert1+p^c\vert=\infty,\] where $\mathcal G$ is defined as in the lemma. Note that we have used the fact that $\sigma_c$ is multiplicative. 
\end{proof}   
\begin{theorem} \label{Thm2.2}
If $a<0$, then $R(c)$ has no isolated points. If $a=0$ and $c\neq 0$, then $R(c)$ is dense in $\mathbb C$ (and, therefore, has no isolated points). If $a>0$ or $c=0$, then every point of $R(c)$ is an isolated point of $R(c)$.  
\end{theorem} 
\begin{proof} 
First, suppose $a<0$, and let $z_0\in R(c)$. Note that $z_0\neq 0$ by Lemma \ref{Lem1.2}. We may write $z_0=\sigma_c(n)$ for some $n\in\mathbb N$. Choose some $\epsilon>0$. To show that $z_0$ is not an isolated point of $R(c)$, we simply need to exhibit a positive integer $N$ such that $0<\vert z_0-\sigma_c(N)\vert<\epsilon$. As $a<0$, we may choose some prime $q>n$ such that $q^a<\displaystyle{\frac{\epsilon}{\vert z_0\vert}}$. Let $N=qn$. As $q$ is relatively prime to $n$, we have \[\vert z_0-\sigma_c(N)\vert=\vert z_0-\sigma_c(q)\sigma_c(n)\vert=\vert z_0\vert\cdot\vert 1-\sigma_c(q)\vert\] \[=\vert z_0\vert\cdot\vert 1-(1+q^c)\vert=q^a\vert z_0\vert<\epsilon.\] This also shows that $\vert z_0-\sigma_c(N)\vert\neq 0$ because $q^a\neq 0$ and $z_0\neq 0$. 
\par 
We now handle the case $a=0$, $c\neq 0$. In this case, $c=bi\neq 0$. We wish to show that $R(c)$ is dense in $\mathbb C$. By Lemma \ref{Lem1.1}, we see that it suffices to prove this claim when $b>0$. Fix some $r>0$ and some $\theta\in(-\pi,\pi]$. Choose $\epsilon\in(0,1/5)$. We wish to exhibit a positive integer $N$ such that $\displaystyle{\vert\sigma_c(N)\vert\in\left(r(1-\epsilon)^2,r(1+\epsilon)^2\right)}$ and $\vert\arg(\sigma_c(N))-\theta+2t\pi\vert<(r+9)\epsilon$ for some integer $t$. This will show that $re^{i\theta}$ is either in $R(c)$ or is a limit point of $R(c)$. Because we may choose $re^{i\theta}$ to be any nonzero complex number, this will prove the assertion that $R(c)$ is dense in $\mathbb C$. By Lemma \ref{Lem1.3}, it is possible to find distinct primes $q_1,q_2,q_3$ such that 
\begin{equation} \label{EqEdit15}
\vert\arg(\sigma_c(q_\ell))-\theta/3\vert<\epsilon
\end{equation} for each $\ell\in\{1,2,3\}$. Note that $\displaystyle{-\frac{\pi}{3}<\frac{\theta}{3}\leq\frac{\pi}{3}}$, so \[\vert\arg(\sigma_c(q_\ell))\vert<\displaystyle{\frac{\pi}{3}+\epsilon<\frac{\pi}{3}+\frac 15}\] for each $\ell\in\{1,2,3\}$. For each $\ell\in\{1,2,3\}$, we know that $\sigma_c(q_\ell)$ lies on the circle $\{1+z\in\mathbb C\colon\vert z\vert=1\}$ and that $\vert\arg(\sigma_c(q_\ell))\vert<\displaystyle{\frac{\pi}{3}+\frac 15}$, so it is easy to verify that $\vert\sigma_c(q_\ell)\vert\geq \displaystyle{\frac 35}$. Therefore, one may easily show that \[\log_{1.9}\left(\frac{r}{\vert \sigma_c(q_1)\sigma_c(q_2)\sigma_c(q_3)\vert}\right)<r+3.\] Let $h$ be a positive integer such that \[\log_{1.9}\left(\frac{r}{\vert \sigma_c(q_1)\sigma_c(q_2)\sigma_c(q_3)\vert}\right)<h<r+4.\] If $p$ is any prime such that $\vert\arg(p^c)\vert<\epsilon$, then it is easy to 
see from \eqref{EqEdit1} that 
\begin{equation} \label{EqEdit16}
\vert\sigma_c(p)\vert=\vert 1+p^c\vert=\sqrt{2+2\cos(\arg(p^c))}>\sqrt{2+2\cos(1/5)}>1.9
\end{equation} because $\vert \arg(p^c)\vert<\epsilon<\displaystyle{\frac 15}$.  Lemma \ref{Lem1.3} tells us that it is possible to choose distinct primes $P_1,P_2,\ldots,P_h$ such that 
\begin{equation} \label{EqEdit21}
\vert\arg(P_j^c)\vert<\epsilon
\end{equation} and $P_j\not\in\{q_1,q_2,q_3\}$ for all $j\in\{1,2,\ldots,h\}$. It follows from \eqref{EqEdit16} that \[\prod_{j=1}^h\vert\sigma_c(P_j)\vert>1.9^h.\] Let $P=q_1q_2q_3P_1P_2\cdots P_h$, and let $Q=\displaystyle{\frac{r}{\vert\sigma_c(P)\vert}}$. Using the fact that \[\log_{1.9}\left(\frac{r}{\vert \sigma_c(q_1)\sigma_c(q_2)\sigma_c(q_3)\vert}\right)<h,\] we have \[\vert\sigma_c(P)\vert=\vert\sigma_c(q_1)\sigma_c(q_2)\sigma_c(q_3)\vert\prod_{j=1}^h\vert\sigma_c(P_j)\vert>1.9^h\vert\sigma_c(q_1)\sigma_c(q_2)\sigma_c(q_3)\vert>r.\] Therefore, $0<Q<1$. This implies that there exists a complex number $z_0$ such that $\Im(z_0)>0$, $\vert z_0\vert=\sqrt Q$, and $\vert z_0-1\vert=1$. By Lemma \ref{Lem1.3}, it is possible to choose distinct primes $q_4$ and $q_5$ such that $q_4\nmid P$, $q_5\nmid P$, 
\begin{equation} \label{EqEdit17}
\vert\arg(\sigma_c(q_4))-\arg(z_0)\vert<\epsilon,
\end{equation} 
\begin{equation} \label{EqEdit18}
\vert\arg(\sigma_c(q_5))+\arg(z_0)\vert<\epsilon,
\end{equation}
\begin{equation} \label{EqEdit19}
\sqrt Q(1-\epsilon)<\vert\sigma_c(q_4)\vert<\sqrt Q(1+\epsilon),
\end{equation} and 
\begin{equation} \label{EqEdit20}
\sqrt Q(1-\epsilon)<\vert\sigma_c(q_5)\vert<\sqrt Q(1+\epsilon).
\end{equation} Essentially, the inequalities \eqref{EqEdit17}, \eqref{EqEdit18}, \eqref{EqEdit19}, and \eqref{EqEdit20} serve to ensure that we have chosen $q_4$ and $q_5$ so that $\sigma_c(q_4)$ is sufficiently close to $z_0$ and $\sigma_c(q_5)$ is sufficiently close to $\overline{z_0}$. If we let $N=q_4q_5P$, then \[\vert\sigma_c(N)\vert=\vert\sigma_c(q_4)\vert\vert\sigma_c(q_5)\vert\vert\sigma_c(P)\vert\] \[=\vert\sigma_c(q_4)\vert\vert\sigma_c(q_5)\vert\frac{r}{Q}\in\left(r(1-\epsilon)^2,r(1+\epsilon)^2\right),\] where we have used \eqref{EqEdit19} and \eqref{EqEdit20}. Also, there exists some integer $t$ such that \[\vert\arg(\sigma_c(N))-\theta+2t\pi\vert=\left\lvert\sum_{\ell=1}^5\arg(\sigma_c(q_\ell))+\sum_{j=1}^h\arg(\sigma_c(P_j))-\theta\right\rvert\] \[\leq\left(\sum_{\ell=1}^3\vert\arg(\sigma_c(q_\ell))-\theta/3\vert\right)+\vert\arg(\sigma_c(q_4))-\arg(z_0)\vert\] \[+\vert\arg(\sigma_c(q_5))+\arg(z_0)\vert+\sum_{j=1}^h\vert\arg(\sigma_c(P_j))\vert.\] We know from \eqref{EqEdit15}, \eqref{EqEdit17}, and \eqref{EqEdit18} that \[\sum_{\ell=1}^3\vert\arg(\sigma_c(q_\ell))-\theta/3\vert<3\epsilon\] and \[\vert\arg(\sigma_c(q_4))-\arg(z_0)\vert+\vert\arg(\sigma_c(q_5))+\arg(z_0)\vert<2\epsilon.\] Similarly, we know from \eqref{EqEdit21} and the fact that $h<r+4$ that \[\sum_{j=1}^h\vert\arg(\sigma_c(P_j))\vert<\sum_{j=1}^h\epsilon<(r+4)\epsilon.\] Hence, \[\vert\arg(\sigma_c(N))-\theta+2t\pi\vert<(r+9)\epsilon.\] This completes the proof of the fact that $R(c)$ is dense in $\mathbb C$ when $a=0$ and $b\neq 0$. 
\par 
Now, assume $a>0$. Choose some $D>0$, and let $m$ be a positive integer such that $\vert \sigma_c(m)\vert<D$. Suppose $p_0^{\alpha_0}\vert m$ for some prime $p_0$ and some positive integer $\alpha_0$. Let us write $m=p_0^\xi y$, where $\xi$ and $y$ are integers, $\xi\geq\alpha_0$ and $p_0\nmid y$. Note that \[\left\lvert\sigma_c\left(p_0^\xi\right)\right\rvert=\left\lvert\frac{p_0^{(\xi+1)c}-1}{p_0^c-1}
\right\rvert\geq\frac{\left\lvert p_0^{(\xi+1)c}\right\rvert-1}{\vert p_0^c\vert+1}=\frac{p_0^{(\xi+1)a}-1}{p_0^a+1}\geq\frac{p_0^{(\alpha_0+1)a}-1}{p_0^a+1}\] and $\vert\sigma_c(y)\vert\geq\displaystyle{\prod_{p<2^{1/a}}(p^a-1)}$ by Lemma \ref{Lem1.4}. For the sake of brevity, 
let $M=\displaystyle{\prod_{p<2^{1/a}}(p^a-1)}$. Then \[D>\vert\sigma_c(m)\vert\geq M\frac{p_0^{(\alpha_0+1)a}-1}{p_0^a+1}.\] If we fix $p_0$, we see that $\alpha_0$ must be 
bounded above. Similarly, if we fix $\alpha_0$, we see that $p_0$ must be bounded above. Consequently, there are only finitely many prime powers $p_0^{\alpha_0}$ that can divide $m$. This implies that there are only finitely many positive integers $m$ such that $\vert\sigma_c(m)\vert<D$. It follows that any disk in the complex plane contains finitely many points of $R(c)$, so every point of $R(c)$ is an isolated point of $R(c)$.     
\par 
The final case we have to consider is when $c=0$. This is easy because $\sigma_0(n)$ is simply the number of divisors of $n$. We see that $R(0)=\mathbb{N}$, so every point in $R(0)$ is an isolated point of $R(0)$. 
\end{proof} 
\section{Closures}
The third question posed in the introduction proves to be more difficult than the first two. For one thing, the first part of the question is fairly open-ended. What exactly would we like to know about the sets $\overline{R(c)}$? In order to ask more specific and interesting questions, we wish to gain a bit of basic information. First of all, if $a>0$ or if $c=0$, then there is not much use in inquiring about the set $\overline{R(c)}$ because this set is the same as $R(c)$ by Theorem \ref{Thm2.2}. If $b=0$ (so that $c$ is real), then $R(c)$ is a subset of $\mathbb R$. In that case, there are many fascinating questions we may ask. For example, as mentioned in the introduction, the author has classified those real $c$ for which $\overline{R(c)}$ is a single closed interval \cite{Defant14}. Here, however, we will not pay too close attention to the sets $\overline{R(c)}$ for real $c$. 
\par 
For the moment, let us streamline our attention toward the sets $\overline{R(c)}$ when $-1\leq a\leq 0$ and $b\neq 0$. By Theorem \ref{Thm2.1}, these sets are not bounded. At the same time, Theorem \ref{Thm2.2} tells us that $R(c)$ has no isolated points for such values of $c$, so we might hope to obtain more interesting sets than those that arise when $a>0$. In fact, if we recall the second part of the third question posed in the introduction and decide to embark on a quest to find those complex $c$ for which $\overline{R(c)}=\mathbb C$, then we need only consider the case $-1\leq a\leq 0$, $b\neq 0$. That is, Theorems \ref{Thm2.1} and \ref{Thm2.2} imply that if $\overline{R(c)}=\mathbb C$, then \begin{equation}\label{EqEdit3} 
-1\leq a\leq 0\hspace{.5cm}\text{and}\hspace{.5cm}b\neq 0. 
\end{equation} 
\par 
To begin this quest, we define a set $\Theta(c)\subseteq\mathbb R$ for each complex $c$ by $\Theta(c)=\{\arg(\sigma_c(n))\colon n\in\mathbb N, \sigma_c(n)\neq 0\}$. The set $\Theta(c)$ is simply the set of arguments of nonzero points in $R(c)$. Clearly, if $\overline{R(c)}=\mathbb C$, then $\overline{\Theta(c)}=[-\pi,\pi]$. In fact, the truth of the converse of this assertion (in the case $-1\leq a\leq 0$) will allow us to deduce our main result. First, we need the following lemma. 
\begin{lemma} \label{Lem3.1} 
If $-1\leq a<0$ and $b\neq 0$, then $\overline{\Theta(c)}=[-\pi,\pi]$. 
\end{lemma}   
\begin{proof} 
Choose some $\theta,\epsilon>0$. We will find a positive integer $n$ such that $\theta-\epsilon<\arg(\sigma_c(n))+2t\pi<\theta+\epsilon$ for some integer $t$, and this will prove the claim. Consider the sets $J_k$ and $\mathcal J$ that were defined in Lemma \ref{Lem1.8}, and let $\mathcal J\cap\mathbb P=\{s_1,s_2,\ldots\}$, where $s_1<s_2<\cdots$. Pick some positive integer $k$. Note that $\vert s_k^c\vert=s_k^a$. Because Lemma \ref{Lem1.8} tells us that $\displaystyle{\frac{\pi}{4}}\leq \arg(s_k^c)\leq\displaystyle{\frac{3\pi}{4}}$, it is not difficult to see from \eqref{EqEdit2} that \[\arg(\sigma_c(s_k))=\arg(1+s_k^c)\geq\tan^{-1}\left(\frac{s_k^a/\sqrt2}{1+s_k^a/\sqrt2}\right).\] Using the fact that $\displaystyle{\frac{1}{s_k}\leq s_k^a<1}$ along with the inequality $\tan^{-1}x\geq x(1-x)$, which holds for all real $x$, we have \[\tan^{-1}\left(\frac{s_k^a/\sqrt2}{1+s_k^a/\sqrt2}\right)\geq \frac{s_k^a/\sqrt2}{1+s_k^a/\sqrt2}\left(1-\frac{s_k^a/\sqrt2}{1+s_k^a/\sqrt2}\right)\] \[=\frac{s_k^a\sqrt2}{(\sqrt2+s_k^a)^2}>\frac{\sqrt2}{(\sqrt2+1)^2}\frac{1}{s_k}.\] It then follows from Lemma \ref{Lem1.8} that $\displaystyle{\sum_{j=1}^\infty\arg(\sigma_c(s_j))=\infty}$. Because $a<0$, $\displaystyle{\lim_{j\rightarrow\infty}\arg(\sigma_c(s_j))=\lim_{j\rightarrow\infty}\arg(1+s_j^a)=0}$. Let $K$ be a positive integer such that $\arg(\sigma_c(s_j))<\epsilon$ for all $j\geq K$. As $\displaystyle{\sum_{j=K}^\infty\arg(\sigma_c(s_j))=\infty}$, there exists some integer $M\geq K$ such that $\theta-\epsilon<\displaystyle{\sum_{j=K}^M\arg(\sigma_c(s_j))<\theta+\epsilon}$. Setting $n=\displaystyle{\prod_{j=K}^Ms_j}$, we have $\displaystyle{\arg(\sigma_c(n))=\sum_{j=K}^M\arg(\sigma_c(s_j))-2t\pi}$ for some integer $t$, from which we obtain the desired inequalities $\theta-\epsilon<\arg(\sigma_c(n))+2t\pi<\theta+\epsilon$.  
\end{proof} 
\begin{theorem} \label{Thm3.1} 
The set $R(c)$ is dense in $\mathbb C$ if and only if $-1\leq a\leq 0$ and $b\neq 0$.   
\end{theorem}
\begin{proof} 
We have seen in \eqref{EqEdit3} that $-1\leq a\leq 0$ and $b\neq 0$ if $\overline{R(c)}=\mathbb C$, so we now wish to prove the converse. If $a=0$ and $b\neq 0$, then we know from Theorem \ref{Thm2.2} that $R(c)$ is dense in $\mathbb C$. Therefore, let us assume that $-1\leq a<0$. With the help of Lemma \ref{Lem1.1}, we may also assume that $b>0$. Fix $c$ (with $-1\leq a<0<b$), and choose some $\theta\in\Theta(c)$ and some $r>0$. Let $n$ be a positive integer such that $\arg(\sigma_c(n))=\theta$. We wish to show that \begin{equation}\label{EqEdit4}
re^{i\theta}\text{ is a limit point of }R(c).
\end{equation} If we can accomplish this goal, then we will know that any arbitrary nonzero complex number $z$ is a limit point of $R(c)$. Indeed, Lemma \ref{Lem3.1} allows us to choose $\theta$ arbitrarily close to $\arg(z)$, and we may set $r=\vert z\vert$. Since we may choose $r$ arbitrarily small, it will then follow that $0$ is also a limit point of $R(c)$ so that $\overline{R(c)}=\mathbb C$. \par 
If $re^{i\theta}\in R(c)$, then we are done proving \eqref{EqEdit4} by Theorem \ref{Thm2.2}. Therefore, let us assume that $re^{i\theta}\not\in R(c)$. In particular, $r\neq\vert\sigma_c(n)\vert$ because $re^{i\theta}\neq\sigma_c(n)$. Choose some $\beta,\epsilon>0$ with $\beta<\displaystyle{\frac{\pi}{4}}$. If $r>\vert\sigma_c(n)\vert$, let \[\Omega=2\sin\beta\log\left(\frac{re^\epsilon}{\vert\sigma_c(n)\vert}\right);\] if $r<\vert\sigma_c(n)\vert$, let \[\Omega=\frac{2\sin\beta}{2-\sqrt2}\log\left(\frac{\vert\sigma_c(n)\vert e^\epsilon}{r}\right).\] We will produce a positive integer $N$ such that \begin{equation}\label{EqEdit5} 
\vert r-\vert\sigma_c(N)\vert\vert<r(e^\epsilon-1)\text{ and }\vert\theta-\arg(\sigma_c(N))+2t\pi\vert<\Omega 
\end{equation} for some integer $t$. As $c,n,$ and $r$ are fixed, we may make $\vert re^{i\theta}-\sigma_c(N)\vert$ as small as we wish by initially choosing sufficiently small values of $\beta$ and $\epsilon$ (this is because $r(e^\epsilon-1)\rightarrow 0$ and $\Omega\rightarrow 0$ as $\epsilon\rightarrow 0$ and $\beta\rightarrow 0$). Therefore, the construction of such an integer $N$ will prove \eqref{EqEdit4} (technically, we must insist that $re^{i\theta}\neq\sigma_c(N)$, but this follows from our assumption that $re^{i\theta}\not\in R(c)$).  
\par 
Let us define $G_k$, $H_k$, $\mathcal G$, and $\mathcal H$ as in Lemmas \ref{Lem1.6} and \ref{Lem1.7}. We will also write $\mathcal G\cap\mathbb{P}=\{q_1,q_2,\ldots\}$ and $\mathcal H\cap\mathbb{P}=\{w_1,w_2,\ldots\}$, where $q_1<q_2<\cdots$ and $w_1<w_2<\cdots$. Because $a<0$, we have $q_k^c\rightarrow 0$ and $w_k^c\rightarrow 0$ as $k\rightarrow\infty$. Thus, $\log\vert1+q_k^c\vert<\epsilon$ and  $\log\vert1+w_k^c\vert>-\epsilon$ for all sufficiently large integers $k$. Let us fix some positive integer $K$ large enough so that for all integers $k\geq K$, we have $\log\vert1+q_k^c\vert<\epsilon$, $\log\vert1+w_k^c\vert>-\epsilon$, $q_k\nmid n$, $w_k\nmid n$,  $q_k^a<\sqrt2-1$, and $w_k^a<\sqrt2-1$. 
\par 
Suppose $r>\vert\sigma_c(n)\vert$. Because $\log\vert1+q_k^c\vert<\epsilon$ for all $k\geq K$ and \\ $\displaystyle{\sum_{k=K}^\infty\log\vert1+q_k^c\vert=\infty}$ by Lemma \ref{Lem1.6}, 
there must exist some integer $M\geq K$ such that \[\log\left(\frac{r}{\vert{\sigma_c(n)\vert}}\right)-\epsilon<\sum_{k=K}^M\log\vert1+q_k^c\vert<\log\left(\frac{r}{\vert\sigma_c(n)\vert}\right)+\epsilon.\] This yields the inequalities
\begin{equation} \label{EqEdit22}
re^{-\epsilon}<\vert\sigma_c(n)\vert\prod_{k=K}^M\vert 1+q_k^c\vert<re^{\epsilon}.
\end{equation} Let $\displaystyle{N=n\prod_{k=K}^Mq_k}$. Because $q_k$ is relatively prime to $n$ for all $k\geq K$, we have \[\sigma_c(N)=\sigma_c(n)\prod_{k=K}^M\sigma_c(q_k)=\sigma_c(n)\prod_{k=K}^M(1+q_k^c).\] This implies that $\vert\sigma_c(N)\vert=\vert\sigma_c(n)\vert\displaystyle{\prod_{k=K}^M\vert1+q_k^c\vert\in(re^{-\epsilon},re^\epsilon)}$, so \[\vert r-\vert\sigma_c(N)\vert\vert<\max\{r-re^{-\epsilon},re^{\epsilon}-r\}=re^{\epsilon}-r=r(e^\epsilon-1),\] as desired in \eqref{EqEdit5}. 
\par 
We now prove that $\vert\theta-\arg(\sigma_c(N))+2t\pi\vert<\Omega$ for some integer $t$. We have \[\arg(\sigma_c(N))=\arg(\sigma_c(n))+\sum_{k=K}^M\arg(1+q_k^c)+2t\pi=\theta+\sum_{k=K}^M\arg(1+q_k^c)+2t\pi\] for some integer $t$. Therefore, \begin{equation}\label{EqEdit7} \vert\theta-\arg(\sigma_c(N))+2t\pi\vert=\left\lvert\sum_{k=K}^M\arg(1+q_k^c)\right\rvert\leq\sum_{k=K}^M\vert\arg(1+q_k^c)\vert. 
\end{equation} Let us fix some $k\in\{K,K+1,\ldots,M\}$. Because $q_k\in\mathcal G$, we know from Lemma \ref{Lem1.6} that $-\beta\leq\arg(q_k^c)\leq\beta$. Also, $\vert q_k^c\vert=q_k^a$. Using \eqref{EqEdit2}, we have \begin{equation}\label{EqEdit8} \vert\arg(1+q_k^c)\vert\leq\tan^{-1}\left(\frac{q_k^a\sin\beta}{1+q_k^a\cos\beta}\right)<\frac{q_k^a\sin\beta}{1+q_k^a\cos\beta}<q_k^a\sin\beta. 
\end{equation} If we write $\eta_k=\displaystyle{\sqrt{1+2q_k^a\cos\beta+q_k^{2a}}}$, then Lemma \ref{Lem1.6} tells us that \begin{equation}\label{EqEdit9} \eta_k\leq\vert 1+q_k^c\vert. 
\end{equation} Recall that we chose $K$ large enough to ensure that $q_k^a<\sqrt{2}-1$. Also, $\cos\beta>\displaystyle{\frac{1}{\sqrt 2}}$ because we chose $\displaystyle{\beta<\frac{\pi}{4}}$. We have $q_k^{2a}<\sqrt{2}q_k^a-q_k^a$, so \[q_k^a<\sqrt{2}q_k^a-q_k^{2a}=\sqrt{2}q_k^a-\frac{1}{2}(\sqrt{2}q_k^a)^2<\log(1+\sqrt{2}q_k^a)\] \[<\log(1+2q_k^a\cos\beta)<2\log\eta_k,\] where we have used the fact that $x-\frac 12x^2<\log(1+x)$ for all $x>0$. This shows that \begin{equation}\label{EqEdit10} 
\frac{q_k^a}{\log\eta_k}<2. 
\end{equation} From the inequalities \eqref{EqEdit7}, \eqref{EqEdit8}, \eqref{EqEdit9}, and \eqref{EqEdit10}, we get \[\vert\theta-\arg(\sigma_c(N))+2t\pi\vert\leq\sum_{k=K}^M\vert\arg(1+q_k^c)\vert<\sum_{k=K}^Mq_k^a\sin\beta\] \[=\sin\beta\log\prod_{k=K}^M\eta_k^{q_k^a/\log\eta_k}<\sin\beta\log\prod_{k=K}^M\eta_k^2\] \[=2\sin\beta\log\prod_{k=K}^M\eta_k\leq2\sin\beta\log\prod_{k=K}^M\vert1+q_k^c\vert.\] We now use the fact that $\displaystyle{\vert\sigma_c(n)\vert\prod_{k=K}^M\vert1+q_k^c\vert<re^\epsilon}$ (obtained in \eqref{EqEdit22}) to conclude that \[\vert\theta-\arg(\sigma_c(N))+2t\pi\vert<2\sin\beta\log\left(\frac{re^\epsilon}{\vert\sigma_c(n)\vert}\right)=\Omega.\] This completes the proof of \eqref{EqEdit5} in the case $r>\vert\sigma_c(n)\vert$. 
\par 
Let us now assume $r<\vert\sigma_c(n)\vert$.  The proof of this case is quite similar to the previous case. Because $\log\vert1+w_k^c\vert>-\epsilon$ for all $k\geq K$ and $\displaystyle{\sum_{k=K}^\infty\log\vert1+w_k^c\vert=-\infty}$ by Lemma \ref{Lem1.7}, 
there must exist some integer $W\geq K$ such that \[\log\left(\frac{r}{\vert{\sigma_c(n)\vert}}\right)-\epsilon<\sum_{k=K}^W\log\vert1+w_k^c\vert<\log\left(\frac{r}{\vert\sigma_c(n)\vert}\right)+\epsilon.\] This 
yields the inequalities 
\begin{equation} \label{EqEdit23} 
re^{-\epsilon}<\vert\sigma_c(n)\vert\prod_{k=K}^W\vert 1+w_k^c\vert<re^{\epsilon}.
\end{equation} In this case, we will let $\displaystyle{N=n\prod_{k=K}^Ww_k}$. Because $w_k$ is relatively prime to $n$ for all $k\geq K$, we have \[\sigma_c(N)=\sigma_c(n)\prod_{k=K}^W\sigma_c(w_k)=\sigma_c(n)\prod_{k=K}^W(1+w_k^c).\] This shows that $\vert\sigma_c(N)\vert=\vert\sigma_c(n)\vert\displaystyle{\prod_{k=K}^W\vert1+w_k^c\vert\in(re^{-\epsilon},re^\epsilon)}$, so \\ $\vert r-\vert\sigma_c(N)\vert\vert<r(e^\epsilon-1)$ once again. \par 
Now, \[\arg(\sigma_c(N))=\arg(\sigma_c(n))+\sum_{k=K}^W\arg(1+w_k^c)+2t\pi\] \[=\theta+\sum_{k=K}^W\arg(1+w_k^c)+2t\pi\] for some integer $t$, so 
\begin{equation}\label{EqEdit11} 
\vert\theta-\arg(\sigma_c(N))+2t\pi\vert=\left\lvert\sum_{k=K}^W\arg(1+w_k^c)\right\rvert\leq\sum_{k=K}^W\vert\arg(1+w_k^c)\vert. 
\end{equation} 
Let us fix some $k\in\{K,K+1,\ldots,W\}$. Because $w_k\in\mathcal H$, we know from Lemma \ref{Lem1.7} that $\arg(w_k^c)\leq-\pi+\beta$ or $\arg(w_k^c)\geq\pi-\beta$.
Recalling that we chose $K$ large enough to guarantee $w_k^a<\sqrt2-1$, we see by Lemma \ref{Lem1.9} that 
\begin{equation}\label{EqEdit12} \vert\arg(1+w_k^c)\vert<\frac{\sin\beta}{2-\sqrt 
2}w_k^a. 
\end{equation} 
In addition, $\sqrt2 w_k^a-w_k^{2a}>w_k^a$. Because $\cos\beta\geq\displaystyle{\frac{1}{\sqrt 2}}$, we have $2w_k^a\cos\beta-w_k^{2a}>w_k^a$. If we write $\displaystyle{\mu_k=\sqrt{1-2w_k^a\cos\beta+w_k^{2a}}}$, then Lemma \ref{Lem1.7} tells us that \begin{equation}\label{EqEdit13} 
\vert 1+w_k^c\vert\leq\mu_k. 
\end{equation} 
We also see that $\mu_k<1$, so $-\log\mu_k>0$. Using the inequality $x\leq-\log(1-x)$, which 
holds for all $x<1$, we have \[w_k^a<2w_k^a\cos\beta-w_k^{2a}\leq-\log(1-(2w_k^a\cos\beta-w_k^{2a}))=-2\log\mu_k.\] Thus, \begin{equation}\label{EqEdit14} -\frac{w_k^a}{\log\mu_k}<2. 
\end{equation} 
Combining inequalities \eqref{EqEdit11}, \eqref{EqEdit12}, \eqref{EqEdit13}, and \eqref{EqEdit14}, we get \[\vert\theta-\arg(\sigma_c(N))+2t\pi\vert\leq\sum_{k=K}^W\vert\arg(1+w_k^c)\vert<\frac{\sin\beta}{2-\sqrt2}\sum_{k=K}^Ww_k^a\] \[=\frac{\sin\beta}{2-\sqrt2}\log\prod_{k=K}^W(1/\mu_k)^{-w_k^a/\log\mu_k}<\frac{\sin\beta}{2-\sqrt2}\log\prod_{k=K}^W(1/\mu_k)^2\] \[=\frac{2\sin\beta}{2-\sqrt2}\log\prod_{k=K}^W(1/\mu_k)\leq\frac{2\sin\beta}{2-\sqrt2}\log\prod_{k=K}^W\vert1+w_k^c\vert^{-1}.\] Using the fact that $\displaystyle{\vert\sigma_c(n)\vert\prod_{k=K}^W\vert1+w_k^c\vert>re^{-\epsilon}}$ (see \eqref{EqEdit23}), we have \[\prod_{k=K}^W\vert1+w_k^c\vert^{-1}<\frac{\vert\sigma_c(n)\vert e^\epsilon}{r}.\] Hence, we conclude that \[\vert\theta-\arg(\sigma_c(N))+2t\pi\vert<\frac{2\sin\beta}{2-\sqrt2}\log\left(\frac{\vert\sigma_c(n)\vert e^\epsilon}{r}\right)=\Omega.\] This completes the proof of \eqref{EqEdit5} in the case $r<\vert\sigma_c(n)\vert$. 
\end{proof} 
Recall the section of the proof of Theorem \ref{Thm2.2} in which we proved that $\overline{R(c)}=\mathbb C$ whenever $a=0$ and $b\neq 0$. We proved this fact by describing the construction of an integer $N$ and showing that we could make $\sigma_c(N)$ arbitrarily close to any predetermined complex number. The observant reader may have noticed that the integer $N$ that we constructed was squarefree. Furthermore, we constructed $N$ using primes that could have been arbitrarily large. In other words, there was never an upper bound on the sizes of the required primes. The same observations are true of the proofs of Lemma \ref{Lem3.1} and Theorem \ref{Thm3.1}. In particular, we chose $\theta$ in the proof of Theorem \ref{Thm3.1} to be an arbitrary element of $\Theta(c)$, and we could have chosen $n$ to be a squarefree number with large prime divisors. Therefore, we obtain the following corollary. 
\begin{corollary} \label{Cor3.1} 
Let $D$ be a positive integer, and let $E(D)$ be the set of squarefree integers whose prime factors are all greater than $D$. Suppose $-1\leq a\leq 0$ and $b\neq 0$. Let $\sigma_c(E(D))=\{\sigma_c(n)\in\mathbb C\colon n\in E(D)\}$. Then $\sigma_c(E(D))$ is dense in $\mathbb C$. 
\end{corollary} 
We now focus on the sets $R(c)$ that arise when $a<-1$. We saw in Theorem \ref{Thm2.1} that these sets are bounded, so it is natural to inquire about the shapes of their closures. We will prove two theorems in order to provide a taste of the questions that one might wish to ask about these sets. The first theorem is very straightforward, but leads to an interesting question. 
\begin{theorem} \label{Thm3.2}
If $a<-1$, then $\displaystyle{\sum_{n\leq x}\sigma_c(n)=x\zeta(1-c)+O(1)}$. 
\end{theorem} 
\begin{proof} 
As $x\rightarrow\infty$, we have \[\sum_{n\leq x}\sigma_c(n)=\sum_{n\leq x}\sum_{km=n}k^c=\sum_{k\leq x}\sum_{m\leq \frac xk}k^c=\sum_{k\leq x}\left\lfloor\frac{x}{k}\right\rfloor k^c\] \[=x\sum_{k\leq x}k^{c-1}+O(1)\sum_{k\leq x}k^c=x\zeta(1-c)-x\sum_{k>x}k^{c-1}+O(1).\] Now, \[\left\lvert x\sum_{k>x}k^{c-1}\right\rvert\leq x\sum_{k>x}\vert k^{c-1}\vert=x\sum_{k>x}k^{a-1}\leq x\int_{x-1}^\infty t^{a-1}\,dt=\frac{x}{\vert a\vert(x-1)^{\vert a\vert}}=o(1),\] and the result follows. 
\end{proof} 
\begin{corollary} \label{Cor3.2} 
If $a<-1$, then $\displaystyle{\lim_{x\rightarrow\infty}\frac 1x\sum_{n\leq x}\sigma_c(n)=\zeta(1-c)}$. 
\end{corollary} 
Suppose $a<-1$. Since $\overline{R(c)}$ is Lebesgue measurable (because it is closed), we may define $\mathcal A(c)$ to be the area of $\overline{R(c)}$ in the complex plane. Provided $\mathcal A(c)\neq 0$, we may let \[C(c)=\frac{1}{\mathcal A(c)}\iint_{\overline{R(c)}}z\,dA\] be the centroid of $\overline{R(c)}$ (where $dA=d(\Re(z))\,d(\Im(z))$ represents a differential area element in the complex plane). Corollary \ref{Cor3.2} endorses $\zeta(1-c)$ as a potential candidate for $C(c)$, although it certainly does not provide a proof. Unfortunately, a rigorous determination of the value of $C(c)$ seems to require knowledge of the shape of $\overline{R(c)}$ because of the apparent necessity of calculating the integrals involved in the definitions of $\mathcal A(c)$ and $C(c)$. The limit in Corollary \ref{Cor3.2} does not appear to be too useful for these purposes because its value depends on the ordering of $\mathbb N$. That is, if we let $m_1,m_2,\ldots$ be some enumeration of the positive integers, then it could very well be the case that $\displaystyle{\lim_{x\rightarrow\infty}\frac{1}{x}\sum_{n\leq x}\sigma_c(m_n)\neq\zeta(1-c)}$ (or that this limit does not exist). Hence, for now, we will let $C(c)$ be.  
\par
It seems natural to ask for the values of $c$ with $a<-1$ for which the sets $\overline{R(c)}$ are connected. The following theorem will show that if $a$ is sufficiently negative (meaning negative and sufficiently large in absolute value), then $\overline{R(c)}$ is separated. In fact, the theorem states that for any positive integer $N$, if $a$ is sufficiently negative, then $\overline{R(c)}$ is a disjoint union of at least $N$ closed sets. This is somewhat unintuitive since, for each $n\in\mathbb N$, $\sigma_c(n)\rightarrow 1$ as $a\rightarrow-\infty$. In other words, the sets $\overline{R(c)}$ ``shrink" while becoming ``more separated" as $a\rightarrow-\infty$. 
\begin{theorem} \label{Thm3.3} 
For each positive integer $N$, there exists a real number $\rho_N\leq-1$ such that if $a<\rho_N$, then $\overline{R(c)}$ is a union of at least $N$ disjoint closed sets. 
\end{theorem} 
\begin{proof} 
We will assume that $a<-1$ throughout this proof. Note that we may set $\rho_1=-1$. We will let $\sigma_c(T)$ denote the image under $\sigma_c$ of a set $T$ of positive integers. 
For each $j\in\mathbb N$, let $p_j$ denote the $j^{th}$ prime number. Let $V_j$ be the set of positive integers that are not divisible by any of the first $j$ primes, and let $S_j$ denote the set of positive integers $n$ such that the smallest prime divisor of $n$ is $p_j$. In symbols, $V_j=\{n\in\mathbb N\colon p_\ell\nmid n\:\forall\:\ell\in\{1,2,\ldots,j\}\}$, and $S_j=V_{j-1}\cap p_j\mathbb N$, where we convene to let $V_0=\mathbb N$. Choose some integer $k$. We will show that there exists a real number $\tau_k\leq -1$ and a function $\delta_k\colon(-\infty,\tau_k)\rightarrow(0,\infty)$ such that if $a<\tau_k$, then $\vert\sigma_c(m)-\sigma_c(n)\vert>\delta_k(a)$ for all $m\in S_k$ and $n\in V_k$. This will show that $\overline{\sigma_c(S_k)}$ is disjoint from $\overline{\sigma_c(V_k)}$ whenever $a<\tau_k$. In particular, it will follow that  $\overline{\sigma_c(S_k)}$ is disjoint from $\overline{\sigma_c(S_j)}$ for all integers $j>k$ whenever $a<\tau_k$. Setting $\rho_{k+1}=\min\{\tau_1,\tau_2,\ldots,\tau_k\}$, we will see that if $a<\rho_{k+1}$, then $\overline{\sigma_c(S_1)},\overline{\sigma_c(S_2)},\ldots,\overline{\sigma_c(S_k)},\overline{\sigma_c(V_k)}$ are $k+1$ disjoint closed sets whose union is $\overline{R(c)}$.
\par 
Choose some $m\in S_k$ and $n\in V_k$. Note that we may write $m=p_k^th$, where $t$ is a positive integer and $h\in V_k$. If we let $\displaystyle{L_k(x)=\sum_{v\in V_k\backslash\{1\}}v^x}$ for all $x<-1$, then we have \[\vert\sigma_c(n)-1\vert=\left\lvert\sum_{\substack{d\vert n \\ d>1}} d^c\right\rvert\leq\sum_{\substack{d\vert n \\ d>1}}\vert d^c\vert=\sum_{\substack{d\vert n \\ d>1}}d^a<L_k(a).\] Similarly, $\vert\sigma_c(h)-1\vert<L_k(a)$. Observe that since $p_{k+1}$ is the smallest element of $V_k\setminus\{1\}$, we have \[L_k(x)<\sum_{v=p_{k+1}}^\infty v^x=p_{k+1}^x+\sum_{v=p_{k+1}+1}^\infty v^x<p_{k+1}^x+\int_{p_{k+1}}^\infty u^x\,du\] \[=p_{k+1}^x+\frac{p_{k+1}}{\vert x+1\vert}p_{k+1}^x=\left(1+\frac{p_{k+1}}{\vert x+1\vert}\right)p_{k+1}^x.\] 
This implies that  $L_k(x)=O(p_{k+1}^x)$ as $x\rightarrow-\infty$, so \[L_k(x)(p_k^x+p_k^{2x}+2)+p_k^{2x}=O\left((p_{k+1}p_k)^x+\left(p_{k+1}p_k^2\right)^x+2p_{k+1}^x+\left(p_k^2\right)^x\right)\] \[=o(p_k^x).\] Hence, there exists some number $\tau_k\leq-1$ such that \begin{equation} \label{Eq1}
L_k(a)\left(p_k^a+p_k^{2a}+2\right)+p_k^{2a}<p_k^a 
\end{equation} for all $a<\tau_k$. Let us define $\delta_k\colon(-\infty,\tau_k)\rightarrow(0,\infty)$ by 
\[\delta_k(x)=\frac{1}{1+p_k^x}\left(p_k^x-p_k^{2x}-L_k(x)(p_k^x+p_k^{2x}+2)\right).\] If $a<\tau_k$, then \[\frac{\vert\sigma_c(p_k^t)-1\vert-\delta_k(a)}{\vert\sigma_c(p_k^t)\vert+1}=\frac{\left\lvert\left(1-p_k^{(t+1)c}\right)/(1-p_k^c)-1\right\rvert-\delta_k(a)}{\left\lvert\left(1-p_k^{(t+1)c}\right)/(1-p_k^c)\right\rvert+1}\] \[=\frac{\left\lvert \left(1-p_k^{(t+1)c}
\right)-(1-p_k^c)\right\rvert-\delta_k(a)\vert 1-p_k^c\vert}{\left\lvert1-p_k^{(t+1)c}\right\rvert+\vert 1-p_k^c\vert}\] \[=\frac{\left\lvert p_k^c-p_k^{(t+1)c}\right\rvert-\delta_k(a)\vert 1-p_k^c\vert}{\left\lvert1-p_k^{(t+1)c}\right\rvert+\vert 1-p_k^c\vert}\geq\frac{p_k^a-p_k^{(t+1)a}-\delta_k(a)(1+p_k^a)}{p_k^a+p_k^{(t+1)a}+2}\] \[\geq
\frac{p_k^a-p_k^{2a}-\delta_k(a)(1+p_k^a)}{p_k^a+p_k^{2a}+2}=L_k(a),\] so $\displaystyle{\vert\sigma_c(p_k^t)-1\vert- L_k(a)(\vert\sigma_c(p_k^t)\vert+1)\geq \delta_k(a)}$. Therefore, \[\vert\sigma_c(m)-\sigma_c(n)\vert=\vert(\sigma_c(p_k^t)-1)-(\sigma_c(p_k^t)-\sigma_c(m))-(\sigma_c(n)-1)\vert\] \[\geq\vert\sigma_c(p_k^t)-1\vert-\vert\sigma_c(p_k^t)-\sigma_c(m)\vert-\vert\sigma_c(n)-1\vert\] \[=\vert\sigma_c(p_k^t)-1\vert-\vert\sigma_c(p_k^t)\vert\vert \sigma_c(h)-1\vert-\vert\sigma_c(n)-1\vert\] \[>\vert\sigma_c(p_k^t)-1\vert-\vert\sigma_c(p_k^t)\vert L_k(a)-L_k(a)\] \[=\vert\sigma_c(p_k^t)-1\vert-L_k(a)(\vert\sigma_c(p_k^t)\vert+1)\geq\delta_k(a)\] for all $a<\tau_k$. This is what we sought to prove.  
\end{proof} 
The expressions in the proof of Theorem \ref{Thm3.3} get a bit messy, and we probably could have simplified the argument by using more careless estimates. However, we organized the proof in order to give an upper bound for the values of $\tau_k$. That is, we showed that we may let $\tau_k\leq-1$ be any number such that \eqref{Eq1} holds for all $a<\tau_k$. In particular, we have the following corollary. 
\begin{corollary} \label{Cor3.3} 
If $a\leq -3.02$, then $\overline{R(c)}$ is separated. 
\end{corollary} 
\begin{proof} 
Preserve the notation from the proof of Theorem \ref{Thm3.3}. Note that \[L_1(a)=-1+\sum_{2\nmid v}v^a=-1+\zeta(-a)(1-2^a).\] Therefore, when $k=1$, \eqref{Eq1} becomes \[(-1+\zeta(-a)(1-2^a))(2^a+2^{2a}+2)+2^{2a}<2^a.\] We may rewrite this inequality as \[\zeta(-a)(1-2^a)<\frac{2^a-2^{2a}}{2^a+2^{2a}+2}+1.\] Noting that $\displaystyle{\frac{2^a-2^{2a}}{2^a+2^{2a}+2}+1=\frac{2+2^{a+1}}{2^a+2^{2a}+2}}$ and dividing each side of this last inequality by $1-2^a$ yields 
\begin{equation} \label{Eq2} 
\zeta(-a)<\frac{1+2^a}{1-2^{a-1}-2^{3a-1}}.
\end{equation} Therefore, we simply need to show that \eqref{Eq2} holds for all $a<-3.02$. If $a\leq -5$, then $\displaystyle{\frac{2^{a+1}}{-(a+1)}\leq 2^{a-1}}$, so \[\zeta(-a)(1-2^{a-1}-2^{3a-1})\leq \left(1+2^a+\int_2^\infty t^adt\right)(1-2^{a-1})\] \[=\left(1+2^a+\frac{2^{a+1}}{-(a+1)}\right)(1-2^{a-1})\leq 1+2^a+\frac{2^{a+1}}{-(a+1)}-2^{a-1}\leq 1+2^a.\] Hence, \eqref{Eq2} holds for $a\leq -5$. We now show that \eqref{Eq2} holds for $-5<a<-3.02$. The reader who wishes to evade the banalities of the following fairly computational argument may wish to simply plot the values of $\zeta(-a)$ and $\displaystyle{\frac{1+2^a}{1-2^{a-1}-2^{3a-1}}}$ to see that \eqref{Eq2} appears to hold for these values of $a$. 
\par 
Let $F(x)=\displaystyle{\zeta(x)-\frac{1+2^{-x}}{1-2^{-x-1}-2^{-3x-1}}}$ and $I=(3.02,5)$ so that our goal is to prove that $F(x)<0$ for all $x\in I$. In order to reduce the number of necessary computations, we will partition $I$ into the two intervals $I_1=(3.02,3.22]$ and $I_2=(3.22,5)$. For all $x\in I$, we have \[F'(x)=\zeta'(x)+\frac{2^{2x+1}(3\cdot 2^x+3\cdot2^{3x}+2)\log 2}{(2^x-1)^2(2^x+2^{2x+1}+1)^2}\] \[\leq\frac{2^{2x+1}(3\cdot2^{3x}+2^{2x})\log 2}{(2^x-1)^2(2^{2x+1})^2}=\frac{(3\cdot2^x+1)\log 2}{2(2^x-1)^2}.\] Now, if $x\in I_1$, then \[F'(x)\leq\frac{(3\cdot2^x+1)\log 2}{2(2^x-1)^2}<\frac{(3\cdot2^{3.22}+1)\log 2}{2(2^{3.02}-1)^2}<0.2.\] Numerical calculations show that if \[x\in\left\{3.02+\frac{n}{2000}\colon n\in\{0,1,2,\ldots,400\}\right\},\] then $\displaystyle{F(x)<-\frac{1}{10000}}$. Because $F$ is continuous, we see that \[F(x)<-\frac{1}{10000}+0.2\left(\frac{1}{2000}\right)=0\] for all $x\in I_1$. If $x\in I_2$, then \[F'(x)\leq\frac{(3\cdot2^x+1)\log 2}{2(2^x-1)^2}<\frac{(3\cdot2^5+1)\log 2}{2(2^{3.22}-1)^2}<0.5.\] For all \[x\in\left\{3.22+\frac{7n}{500}\colon n\in\{0,1,2,\ldots,127\}\right\},\] numerical calculations show that $\displaystyle{F(x)<-\frac{7}{1000}}$. Because $F$ is continuous, we see that \[F(x)<-\frac{7}{1000}+0.5\left(\frac{7}{500}\right)=0\] for all $x\in I_2$. 
\end{proof}  
\section{Concluding Remarks and Open Problems} 
\begin{figure} 
\begin{center}
\epsfysize=8.6cm \epsfbox{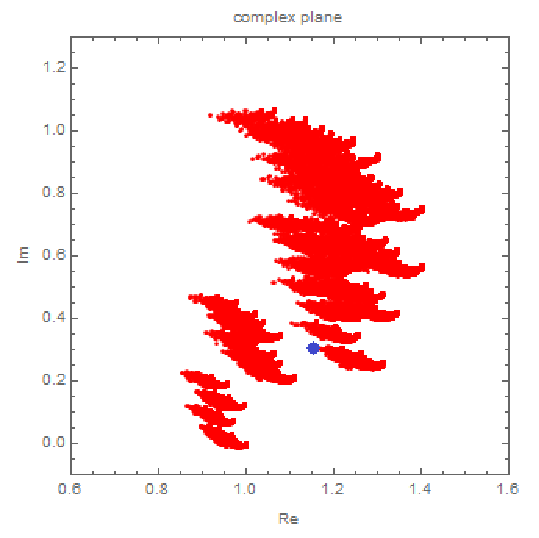}
\\ 
\epsfysize=8.6cm \epsfbox{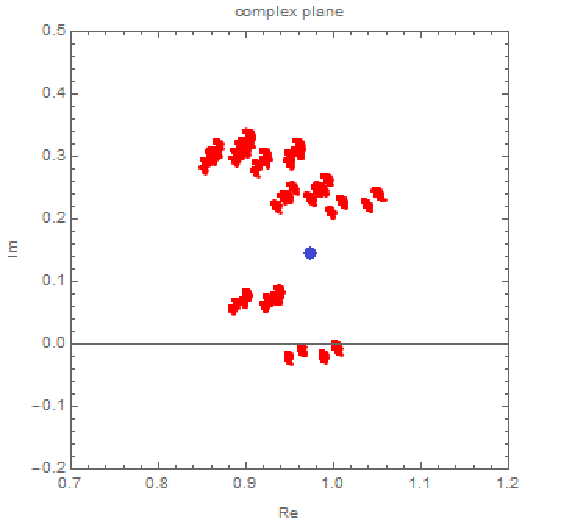}
\caption{Plots of $\sigma_c(n)$ for $1\leq n\leq 10^6$. The top image shows a plot for $c=-1.3+i$, and the bottom shows a plot for $c=-2+2i$. In each image, the blue dot is the point $\zeta(1-c)$.} 
\end{center} 
\end{figure}
We have obtained a decent understanding of the ranges $R(c)$ of the functions $\sigma_c$ for values of $c$ with real part $a\geq-1$, but many problems concerning the sets $R(c)$ that arise when $a<-1$ remain open. For example, it would be quite interesting to determine the values of $c$ for which $\overline{R(c)}$ is connected. Corollary \ref{Cor3.3} provides an initial step in this direction, but there is certainly much work that remains to be done. For example, it seems as though the bound $a<-3.02$ in that corollary is quite weak since the estimates used to derive it are far from optimal. Recall that we derived that bound by showing that if $a<-3.02$, then $\overline{R(c)}$ is the disjoint union of the closed sets $\overline{\sigma_c(S_1)}$ and $\overline{\sigma_c(V_1)}$ defined in the proof of Theorem \ref{Thm3.3}. We remark, however, that it might be more useful to look at the values of $c$ for which $\overline{\sigma_c(V_2)}$ is disjoint from $\overline{\sigma_c(S_1\cup S_2)}$. Indeed, if we confine $c$ to the real axis and decrease $c$ from $-1.5$ to $-2$, we will see that $\overline{R(c)}$ begins as a connected set and then separates into the disjoint union of the two connected sets $\overline{\sigma_c(V_2)}$ and $\overline{\sigma_c(S_1\cup S_2)}$. More formally, one may use the methods described in \cite{Defant14} to show that there exist constants $\eta\approx 1.8877909$ and $\kappa\approx1.9401017$ such that if $c$ is real and $-\kappa<c<-\eta$, then $\overline{R(c)}$ is the disjoint union of the two connected sets $\overline{\sigma_c(V_2)}$ and $\overline{\sigma_c(S_1\cup S_2)}$ (this phenomenon essentially occurs because the largest value of $p_{j+1}/p_j$ occurs when $j=2$). Nonetheless, a full determination of those complex $c$ for which $\overline{R(c)}$ is connected seems to require knowledge about the specific shapes of the sets $\overline{R(c)}$. Using Mathematica to plot points of some of these sets allows one to see the emergence of sets that appear to have certain fractal-like properties (see Figure 1). However, as we can only plot finitely many points, it is difficult to predict the shapes of the full sets $R(c)$ and their closures. 
\par 
We mention once again the open problem of determining the values of $\mathcal A(c)$ and $C(c)$ defined after Corollary \ref{Cor3.2} above.
Observe the point $\zeta(1-c)$ plotted in the top image of Figure 1 (corresponding to $c=-1.3+i$). Upon visual inspection, this point does not appear to be the centroid $C(c)$ of $\overline{R(c)}$. This is likely due to the fact that points of $R(c)$ clustered in the lower part of the image are packed more densely than those in the middle part. Finally, we remark that, while investigating the topics discussed in Section 3, the author found that it would be useful to have good upper and lower bounds for $\Theta(c)$. Thus, we state the derivation of such bounds as an additional potential topic for future research.  
\section{Acknowledgments} 
The author would like to express his gratitude to the anonymous referee who read carefully through this manuscript and gave several helpful suggestions.

\end{document}